\documentclass{amsart}

\usepackage{amssymb}
\usepackage{mathtools}
\mathtoolsset{showonlyrefs}
\usepackage[stretch=10,shrink=10]{microtype}
\usepackage[shortlabels]{enumitem}
\usepackage{tikz}

\usepackage{pgfplots}
\pgfplotsset{compat=1.18}
\usetikzlibrary{arrows.meta}
\usepackage{comment}
\usepackage{nccmath}
\usepackage[normalem]{ulem}

\newcommand{\q}{\mathfrak{q}}

 \newcommand{\f}{\frac}

\renewcommand{\phi}{\varphi}

\newcommand{\E}{\mathbf E}
\renewcommand{\P}{\mathbf P}

\DeclareMathOperator{\Geo}{Geo}
\DeclareMathOperator{\Ber}{Ber}
\DeclareMathOperator{\Bin}{Bin}

\DeclareMathOperator{\RFM}{RFM}
\DeclareMathOperator{\FM}{FM}

\usepackage{hyperref}
\hypersetup{colorlinks,linkcolor={blue},citecolor={olive},urlcolor={brown}}
\usepackage[all]{hypcap}
\usepackage{theoremref}

\newtheorem{theorem}{Theorem}
\newtheorem{lemma}[theorem]{Lemma}

\newtheorem*{lemma*}{Lemma}

\newtheorem{proposition}[theorem]{Proposition}
\newtheorem{corollary}[theorem]{Corollary}

\theoremstyle{definition}

\newtheorem*{remark*}{Remark}

\title{The frog model with death revisited}

\author{Samyah Ahmed} \email{samyah.ahmed@baruchmail.cuny.edu}
\author{Matthew Junge} \email{Matthew.Junge@baruch.cuny.edu}

\begin{document}
\maketitle

\begin{abstract}
We prove that the probability the frog model with death and drift on the $d$-ary tree is recurrent can be made positive and thus is not monotone in the drift parameter.
\end{abstract}

\section{Introduction}
For the \emph{frog model with death} place an active frog at the root of the infinite rooted $d$-ary tree $\mathbb T_d$ and one sleeping frog at the non-root vertices. Active frogs perform simple random walk, but independently die before taking each step with probability $1-q$. Whenever an active frog visits a sleeping frog, the sleeping frog becomes active. We will refer to this process as $\FM(d, q)$. These random activation dynamics have been interpreted as percolation \cite{popov2003frogs}, infection \cite{alves2002shape} and combustion \cite{ramirez2004asymptotic} models.

In \cite{alves2002phase} and subsequent refinements \cite{lebensztayn2005improved, gallo2018frog, lebensztayn2019new, gallo2023critical}, bounds were proven for the phase-transition at which the process survives with positive probability.
A realization of the frog model is \emph{recurrent} if the root is visited by infinitely many active frogs and \emph{transient} if not. Recurrence is a stronger condition than survival and difficult to study. For example, proving that the frog model with $q=1$ on $\mathbb T_2$ is recurrent with positive probability went unsolved for over a decade, and the question is still open for $\mathbb T_3$ and $\mathbb T_4$ \cite{popov2003frogs, hoffman2017recurrence}. 

It is easy to prove that $\FM(d,q)$ is transient almost surely for any $d \geq 2$ and $q <1$ (see \thref{prop:p}). %
Seeking a more recurrent process, \cite{beckman2019frog} introduced a variant of the frog model on $\mathbb T_d$ in which active frogs have a drift $p$ towards to the root. They only considered the case $q=1$ with no death and showed that the process can be made recurrent  almost surely for any $d\geq 2$ so long as $p> 0.4155$. Followup efforts have focused on improving their bound \cite{guo2022minimal, bailey2023critical, mathews2024improved}. In this work we \emph{revisit} the frog model with death by combining  death and drift and asking if this more general model can be made recurrent with positive probability when $q<1$. 

\subsection{Model and results}
Let $\FM(d,q,p)$ be the \emph{frog model with death $1-q$ and drift $p$} on $\mathbb T_d$. The initial conditions, death dynamics, and waking dynamics are as with $\FM(d,q)$. The difference is that active frogs at non-root sites that survive to jump will move towards the root with probability $p$ and away from the root to a uniformly chosen child vertex with probability $1-p$. Active frogs at the root that survive stay in place with probability $p$ or jump away from the root to a uniformly chosen child with probability $1-p$. Note that $\FM(d,q,(d+1)^{-1}) = \FM(d,q)$.

The recurrence and transience behavior with death and drift present is subtle because of competing forces. Increasing the drift makes active frogs more likely to return to the root and to explore distant parts of the tree, both of which ought to promote recurrence. However, since sleeping frogs are only awakened during steps away from the root, excessive drift can cause active frogs to spend most of their lifespan moving towards the root not waking others.

Proving continuity and monotonicity results in the parameters for the frog model is believed to be difficult \cite{fontes2004critical, bailey2023critical, mathews2024improved}. Additionally, since the frog model with death and simple random walks is transient for any $q<1$, it is unclear if there is any recurrent phase at all.
Our main result answers this question in the affirmative. %

\begin{theorem}\thlabel{thm:main}
    $\FM(d,q,p) \text{ is recurrent}$ with positive probability for any $d \geq 2$ for all $(p,q) \in[0.4950, 0.4975] \times [0.9999998,1]$. %
\end{theorem}

We did not try to optimize the size of the rectangle in \thref{thm:main}, but believe that it cannot be made dramatically larger in the $q$-coordinate using our methods. 
Next, we give sufficient conditions for transience.

\begin{proposition}\thlabel{prop:p}
    Given $d\geq 2$ and $q<1$, we have $\FM(d,q,p)$ is transient almost surely if either (a) $q \leq 1/2$ or $p \geq 2-q^{-1}$, or (b) $p < 1/ (1+qd)$. 
\end{proposition}

An unsolved conjecture is proving that the probability $\FM(d,1,p)$ is recurrent is monotone in $p$ \cite{beckman2019frog, bailey2023critical, mathews2024improved}. Our results tells us that $\FM(d,q,p)$ is not necessarily monotone when $q<1$.

\begin{corollary} \thlabel{cor:mono}
    $\P(\FM(d,q,p) \text{ is recurrent})$ is not monotone in $p$ for any fixed  $d \geq 2$ and  $q \in [ 0.9999998,1)$. 
\end{corollary}
Note that \thref{cor:mono} is not meant to be evidence against the monotonicity conjecture from \cite{beckman2019frog}, which we still find plausible. Rather, our result is intended to illustrate that the phase structure when death is present is nuanced. See Figure~\ref{fig:pq} for the phase-diagram. 

\begin{figure}
\centering
\begin{tikzpicture}[xscale=4, yscale=4]
  \draw (0,0) rectangle (1,1);

  \fill[gray!25] (0,0) rectangle (1,0.5);

  \fill[gray!25, domain=0.5:1, variable=\q]
    plot ({1/(1+2*\q)}, \q)
    -- (0,1) -- (0,0.5) -- cycle;

  \fill[gray!25, domain=0.5:1, variable=\q]
    plot ({2 - 1/\q}, \q)
    -- (1,1) -- (1,0.5) -- cycle;

  \filldraw[blue!20, draw=blue, fill=blue, thick]
    (0.47,0.98) rectangle (0.4975,1);

  \foreach \x in {0,1}
    \draw (\x,0) -- (\x,-0.01) node[below, font=\tiny] {\x};

  \foreach \y in {0,1}
    \draw (0,\y) -- (-0.008,\y) node[left, font=\tiny] {\y};

  \node at (0.5,-0.05) {$p$};
  \node at (-0.05,0.5) {$q$};
\end{tikzpicture}

\caption{The $(p,q)$-phase diagram for $\FM(2,q,p)$ implied by our results. In the gray shaded regions the process is almost surely transient. The blue region represents the rectangle $[0.4950,0.4975] \times [0.9999998,1]$ (enlarged to make it visible at this scale) for which the process is recurrent  with positive probability. The white region is not covered by our results. 
} \label{fig:pq}
\end{figure}
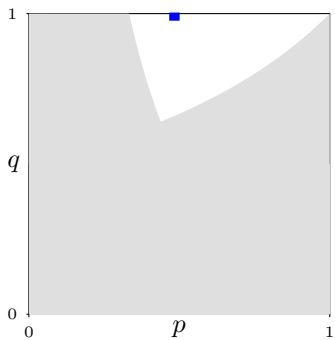

\subsection{Proof Discussion}

The proof of \thref{thm:main} goes by generalizing the method from \cite{beckman2019frog} to include death. Roughly speaking we show that a less recurrent subprocess, the \emph{recursive frog model}, is recurrent with positive probability for the parameter choices in \thref{thm:main}. This particular subprocess has only been used once before. The far more commonly used tool to prove recurrence on trees is a different subprocess known as the \emph{self-similar frog model}. We counted a dozen articles that rely on it. %

We first tried the self-similar frog model as a means to prove recurrence of $\FM(d,q,p)$. However, we were surprised to encounter an issue when death is incorporated. Namely, the Poisson bootstrapping method introduced in \cite{hoffman2016transience}, see \cite{bailey2023critical} for perhaps the most accessible description, appears to fail irreparably. The reason the recursive frog model is more robust to the presence of death is that the analysis only depends on first and second moments of a recursive formula rather than the whole distribution. This is the first instance we are aware of for which the recursive frog model prevails where the self-similar frog model does not, which is an interesting technical feature from our work. 

The proof of \thref{prop:p} shows that simpler models that dominate $\FM(d,q,p)$ are transient for the claimed parameter values. Specifically, we use a branching process and the process with all frogs initially awake. These are standard, elementary arguments that we include because \thref{prop:p} fills in much of the phase diagram and is needed to deduce \thref{cor:mono}.

\subsection{Organization}

In Section~\ref{sec:rfm} we introduce the recursive frog model and describe its properties.
In Section~\ref{sec:proofs} we prove \thref{thm:main}.
Section~\ref{sec:prop} has the proof of \thref{prop:p}.
The appendix contains a lemma regarding the length of random walk excursions.

\section{The recursive frog model}\label{sec:rfm}

We begin this section by defining and adapting the recursive frog model to include death and then state some of its properties.

\subsection{Definition}

Suppose that $0 < p < 1/2$ and let $\rho := p/(1-p)$. Denote the root of $\mathbb T_d$ by $\varnothing$. We will think of the root as the far left of the horizontally oriented tree and refer to frogs moving towards the root as moving \emph{left} and those moving away from the root as moving \emph{right}. The \emph{recursive frog model} $\RFM(d,q,p)$ starts with an active frog at $\varnothing$ and a sleeping frog at each $x \neq \varnothing$. We start with the case $q=1$ so that our definition aligns exactly with \cite{beckman2019frog}. Subsequently, we will explain how to modify for $q<1$. 

In $\RFM(d,1,p)$, the frog at the root follows a non-backtracking ray from the root to infinity sampled uniformly at random. This ray is obtained by including the last site at each distance $n=0,1,\hdots,$ from the root that the initially active frog visits. Since $p<1/2$, a unique such path exists almost surely. When a sleeping frog  is visited it becomes active. Activated frogs initially move left with probability $\rho$ and right with probability $1-\rho$ to a uniformly random child at each step. Frogs that move left to $\varnothing$ are killed (after counting their visit). Once an activated frog moves right it begins moving to uniformly sampled child vertices. Frogs are killed if they move right to an already visited site. 

Since the recursive frog model takes some getting used to, we clarify a bit more about how these paths are sampled and rules enforced. Suppose an activated frog begins at distance $n$ from $\varnothing$. The leftward path is sampled by taking the first vertex at distances $n-1$, $n-2$, ... from the root visited by the active frog's random walk. At some point the walk will either reach $\varnothing$ and be killed, or visit some leftmost site at positive distance from the root. At this point its rightward ray is sampled starting from that site independently, but in the same manner as the frog's started at $\varnothing$. During its rightward descent, if it ever visits a site that it or some other activated frog has visited it dies. Note that ties, where two or more rightward moving frogs arrive at simultaneously to the same site, are broken by choosing one of the frogs uniformly at random to survive and killing the others.

It is proven in \cite[Lemma 2.1]{beckman2019frog} that if $\RFM(d,1,p)$ is recurrent with positive probability then so is $\FM(d,1,p)$. The basic idea is that the paths frogs follow are trimmed and truncated versions of their full simple random walk paths. Hence, $\RFM(d,1,p)$ is less likely to be recurrent. Taking $q<1$ brings an added complication; now trimming the range is not necessarily a monotone operation since steps where the frog may have died are omitted. 

In \thref{lem:biased-walk-bounds} we prove that the total excursion lengths between first and final visits to different distances by a $p$-biased walk on $\mathbb Z$ (and thus $\mathbb T_d$ when only tracking displacement from the root) is stochastically dominated by independent copies of
$$T \sim 2 \Geo(s) + 1 \text{ with } s := (1-2p) \exp\left(-\frac{3(1 - 2p)^2}{2(1 + 4p)} \right).$$
Accordingly, we define $\RFM(d,q,p)$ for $q<1$ to consist of the same paths in $\RFM(d,1,p)$, but with each active frog surviving to take a jump at each step independently with probability
\begin{align}
\q := \E[q^T]= \f{sq}{1-(1-s)q^2}. \label{eq:q}
\end{align}
This gives an upper bound on the probability that the active frog died during some step in its full random walk path that is not counted by its trimmed sub-path used in the recursive frog model.

\subsection{Properties}

Using the stochastic domination from \thref{lem:biased-walk-bounds} and then following the proof of \cite[Lemma 2.1]{beckman2019frog} we conclude that if $\RFM(d,q,p)$ is recurrent with positive probability then so is $\FM(d,q,p)$. Moreover, \cite[Proposition 1.3]{beckman2019frog} observes that root visits by $\RFM(d,1,p)$ are stochastically increasing in $d$. The argument is very similar when death is present. Hence it is enough to prove recurrence for $\RFM(2,q,p)$. We record this conclusion as a lemma.

\begin{lemma} \thlabel{lem:suf}
    $\P(\RFM(2,q,p) \text{ is recurrent}) \leq \P(\FM(d,q,p) \text{ is recurrent})$ for  $d \geq 2$. 
\end{lemma}

Suppose that $t \geq 2$. The recursive frog model earns its name from a recursive equation \cite[Equation 3.3]{beckman2019frog} satisfied by the total number of root visits $V_t = V_t(2,q,p)$ by the recursive frog model on $\mathbb T_2$ with sleeping frogs only to distance $t$ from the root and empty sites beyond.  The equation is still valid after a small modification to account for death. We describe it now.

If the frog at the root lives to take its first step, it will move to a child vertex that we denote by $\varnothing'$. 
If the frog started at the root survives to jump from $\varnothing'$, call the child vertex it jumps to $x$ and the other child vertex $y$. See Figure~\ref{fig:rfm}.

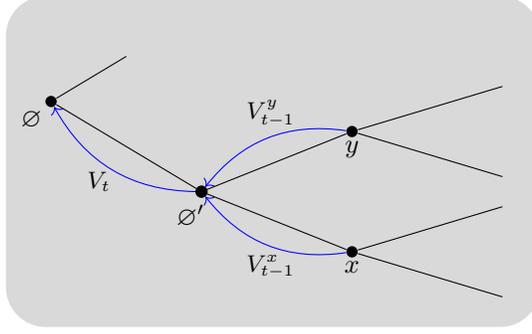
\begin{figure}
    \begin{center}
\begin{tikzpicture}[scale = 2,tv/.style={circle,fill,inner sep=0,
    minimum size=0.15cm,draw}, starttv/.style={circle,fill,inner sep=0,
    minimum size=0.15cm},walklines/.style={very thick, blue,bend left=15}]
\fill[black!15!white,rounded corners=15pt] (-0.3, -1.5) rectangle (3.25, .7);
    \path (0,0) node[starttv] (R) {}
        (1,-0.6) node[tv] (L1) {}
        (1, 0.6)  node (L2) {}
        (2, -0.2) node[circle, fill, inner sep = 0, minimum size = .15cm] (L12) {}
        (2,-1) node[circle, fill, inner sep = 0, minimum size = .15cm] (L11) {}
        (L12)+(1,-0.3) coordinate (L121)
             +(1, 0.3) coordinate (L122)
        (L11)+(1,-0.3) coordinate (L111)
             +(1, 0.3) coordinate (L112);
  \draw
       (R)--(L1);
 \path[every node/.style={font=\sffamily\small}]
       (L1) edge [->,bend left,blue] node[below left = -.1 cm] {\;\;$\textcolor{black}{V_{t}}$} (R)
       (L11) edge [->,bend left,blue] node[below] {\;$\textcolor{black}{V_{t-1}^x}$} (L1)
       (L12) edge [->,bend right,blue] node[above = .03 cm] {\;$\textcolor{black}{V_{t-1}^y}$} (L1);
\node[below left] at (R){$\varnothing$};
\node[below] at (L12){$y$};
\node[below] at (L11){$x$};
\node[below =.05 cm] at (L1){$\varnothing'$\;\;\;};
    \draw (L1)--(L11);
    \draw (R) -- (.5,.3);
       \draw (L1) -- (L12);
  \draw
       (L12)--(L122) (L12)--(L121);

  \draw     (L11)--(L112) (L11)--(L111);

\end{tikzpicture}
    \end{center}
    \caption{The recursive structure of $\RFM(2,q,p)$. Here, $V_{t}$ is the total number of visits to $\varnothing$ with sleeping frogs placed up to distance $t$ from the root. It can be expressed as a binomial thinning of the number of visits to $\varnothing'$ from each of its children $x$ and $y$. Conditional on $x$ and $y$ being visited, these quantities are i.i.d.\ and distributed like $V_{t-1}$.
    }
    \label{fig:rfm}
  \end{figure}

Call the event that $\varnothing'$ is visited by the active frog started at the root $A^{\varnothing'}$. Call the events that $x$ is ever visited by an active frog and that $y$ is ever visited by any active frog $A_t^x$ and $A_t^y$, respectively. Let $V_{t-1}(x)$ and $V_{t-1}(y)$ be independent and identically distributed copies of $V_{t-1}$. Let $X_t \sim \Bin(V_{t-1}(x), \rho \q)$ and $Y_t \sim \Bin(V_{t-1}(y), \rho \q)$,  and $Z \sim \Ber(\rho \q)$. The analogue of \cite[Equation 3.3]{beckman2018asymptotic} is
\begin{align} \label{eq:rec}
V_t = \mathbf 1 \{A^x_{t}\} X_t+\mathbf{1}\{A^y_t\cap A^x_t\}Y_t+\mathbf{1}\{ A^{\varnothing'}\}Z.
\end{align}
Here $Z$ detects whether or not the frog at $z$ is sent back to $\varnothing$. The random variables $X_t$ and $Y_t$ are the number of frogs sent to $\varnothing'$ from the subtrees rooted at $x$ and $y$, respectively, that then move to $\varnothing$. The indicators tell us that these sites are activated.

\begin{lemma} \thlabel{lem:EV}
$\E[V_t] = \rho \q^3(1+\P(A_t^y \mid A_t^x))\E[V_{t-1}]+\rho \q^2$ for all $t \geq 1$.
\end{lemma}

\begin{proof}
Taking the expectation of both sides of \eqref{eq:rec} and computing some simple probabilities gives
\begin{align}
\E[V_t] &= \P(A^x_t) \E[X_t] +\P(A^y_t\cap A^x_t) \E[Y_t ] + \P( A^{\varnothing'}) \E[Z]\\
&= (\P(A^x_t) +\P(A^y_t\cap A^x_t))\rho \q \E[V_{t-1}] + \rho \q^2\\
&= (\q^2 +\P(A^y_t\mid A^x_t) \q^2)\rho \q \E[V_{t-1}] + \rho \q^2.
\end{align}
Factoring out the $\q^2$ from the first term gives the claimed equality.
\end{proof}

Looking ahead to an application of the Paley-Zygmund inequality, we need a bound on a moment ratio.

\begin{lemma}\thlabel{lem:EV2/EV}
If $\rho \q^3(1+\P(A_{t}^y \mid A_{t}^x))>1$ for all large $t$, then there exists $C=C(q,p)>0$ such that  $$\limsup_{t\to \infty} \E[V_t^2]/\E[V_t]^2 \leq C < \infty.$$
\end{lemma}
\begin{proof}
    First we obtain a first order approximation for $\E[V_t^2]$. Squaring both sides of \eqref{eq:rec} gives
\begin{align}
V_{t}^2
&= \mathbf 1\{A^x_{t}\} X_{t}^2
   + \mathbf 1\{A^y_{t}\cap A^x_{t}\} Y_{t}^2
   + \mathbf 1\{A^{\varnothing'}\} Z^2 \\[4pt]
&\qquad\qquad + 2\,\mathbf 1\{A^x_{t}\}\mathbf 1\{A^y_{t}\cap A^x_{t}\}\,X_{t}Y_{t} \\[4pt]
&\qquad\qquad\qquad + 2\,\mathbf 1\{A^x_{t}\}\mathbf 1\{A^{\varnothing'}\}\,X_{t}Z
   + 2\,\mathbf 1\{A^y_{t}\cap A^x_{t}\}\mathbf 1\{A^{\varnothing'}\}\,Y_{t}Z.
\end{align}
Taking expectations, using the fact that $X_{t}$ and $Y_{t}$ are distributed as independent $\Bin(V_{t-1},\rho \q)$ random variables, and following the derivations preceding \cite[Equation (3.8)]{beckman2019frog} gives
\begin{align}
\E[V_{t}^2] &= (\P(A^y_{t}\cap A^x_{t})+\P(A^x_{t})) (\rho \q)^2 \E[V_{t-1}^2]  \\
    &\hspace{4 cm}+  2\P(A^y_{t}\cap A^x_{t})(\rho \q)^2 \E[V_{t-1}]^2+ O(\E[V_{t-1}]) \\
&= \rho^2 \q^4(1+\P(A^y_{t}\mid A^x_{t}))  \E[V_{t-1}^2] +  2\rho^2 \q^4\P(A^y_{t}\mid A^x_{t})\E[V_{t-1}]^2 + O(\E[V_{t-1}]). \label{eq:EV2}
\end{align}

Squaring the equation in \thref{lem:EV} gives 
\begin{align}\E[V_{t}]^2 = \rho^2 \q^6[1+\P(A_{t}^y \mid A_{t}^x)]^2\E[V_{t-1}]^2+O(\E[V_{t-1}]).\label{eq:EV}
\end{align} 
Let $v_{t} = \E[V_{t}^2]/\E[V_{t}]^2$. Referring to \thref{lem:EV}, our hypothesis that $\rho \q^3(1+\P(A_{t}^y \mid A_{t}^x))>1$ ensures that $\E[V_t] \to \infty$. Dividing \eqref{eq:EV2} by \eqref{eq:EV} and simplifying gives
\begin{align}
    v_{t} = \f{ 1} {\q^2(1 + \P(A_{t}^y \mid A_{t}^x))}  v_{t-1} + O(1) 
\end{align}
Our hypothesis $\rho \q^3(1+\P(A_{t}^y \mid A_{t}^x))>1$ ensures that the coefficient of $v_{t-1}$ is less than 1 for large $t$ and thus $\limsup v_t < C < \infty$ for some constant $C$.
\end{proof}

We now have a simple sufficient condition for recurrence with positive probability. 

\begin{lemma} \thlabel{cor:V>0}
    If $\P(A_t^y \mid A_t^x) \ge (\rho \q^3)^{-1} - 1$ for some $t \geq 1$, then $$\P(\textstyle \lim_{t \to \infty} V_t(d,q,p) = \infty) \geq (4C)^{-1}$$ for all $d \geq 2$ with $C$ as in \thref{lem:EV2/EV}. 
\end{lemma}

\begin{proof}
An immediate corollary of \thref{lem:EV} is that $\E[V_t]$ diverges to infinity so long as $\rho \q^3[1+\P(A_{t}^y\mid A_{t}^x)] \geq 1$ for some $t \geq 1$. Note that this also implies that $\rho \q^2[1+\P(A_{t}^y\mid A_{t}^x)] > 1$ as required by  \thref{lem:EV2/EV}. 

Some algebra shows that $\rho \q^3[1+\P(A_{t}^y\mid A_{t}^x)] \geq 1$ is equivalent to the inequality in the corollary's statement.  Using \thref{lem:EV2/EV} and the Paley-Zygmund inequality we deduce that 
\begin{align}
\P(\textstyle \lim V_t(2,q,p) = \infty) &\geq \liminf \P( V_t \geq \E[V_t]/2) \\
&\geq \liminf (4 \E[V_t^2]/\E[V_t]^2 )^{-1} \\
&\geq (4C)^{-1} .
\end{align}
Applying  \thref{lem:suf} gives our claim for all $d\geq 2$. 
\end{proof}

\section{Proof of \thref{thm:main}} \label{sec:proofs}

\thref{cor:V>0} gives a sufficient condition for a positive probability of recurrence. Specifically,  we must show that 
\begin{align}
\P(A_t^y \mid A_t^x) \ge \frac{1}{\rho \q^3} - 1 \label{eq:suf}
\end{align}
for some $t \geq 1$ and choice of $q$ and $p$.  
Recall that $A_t^y$ is the event that the vertex $y$ is eventually visited by an active frog in $\RFM(2,q,p)$. These events are monotonically increasing in $t$ in terms of set containment. 

We use the same scheme as in \cite{beckman2019frog}, but with death incorporated. Since we are conditioning on $A_t^x$, we know that the frog started at the root reached site $x$. It will continue jumping right from $\varnothing'$ to $x$, activating $J_t\sim \Geo(1-\q) \wedge (t-2)$  sleeping frogs before either dying or moving to a section of a tree with no frogs. The sleeping frog at distance $k\geq 0$ from $\varnothing'$ on this ray has probability at least $(\rho\q)^{k} \q(1-\rho)/2$ of jumping directly along the geodesic from its starting vertex to $y$ without dying. Thus, we obtain a lower bound
\begin{align}
\P(A_t^y \mid A_t^x)& \geq 1-  \E \prod_{k=0}^{J_t} \left( 1- \left((\rho \q)^{k} \q \tfrac{1- \rho}{2}  \right) \right)\\
    &\geq  1-  \left(\prod_{k=0}^{t-2}\left( 1- \left((\rho \q)^{k} \tfrac{1- \rho}{2} \q \right) \right) \P(J_t \geq t-2) \right)\\
      &\geq   1-  \left(\prod_{k=0}^{t-2}\left( 1- \left((\rho \q)^{k} \tfrac{1- \rho}{2} \q \right) \right) \q^{t-2} \right). \label{eq:P}
\end{align}

Now, we take concrete values in order to provide an explicit rectangle. Set $t=52$ and suppose that $\rho$ is restricted to the interval $[0.98,0.99]$ and $\q$ is restricted to $[0.9999,1]$. Using the worst case endpoint values and numerical calculations we have
\begin{align}
\P(A_{52}^y \mid A_{52}^x)& \geq 1-  \left(\prod_{k=0}^{50}\left( 1- \left((0.98\times 0.9999)^{k}(.005)(0.9999) \right) \right) 1^{50} \right) \\
&\geq 0.1485.
\end{align}
Since $$\f1 { \rho \q^3} -1 \leq \f 1 { 0.99 \times 0.9999^3} -1 \leq 0.0104,$$ 
we conclude that \eqref{eq:suf} is satisfied for $\rho \in [.98,.99]$ and $\q \in [0.9999,1].$ The definition $\rho = p / (1-p)$  and some simple algebra shows that $p \in [0.4950, 0.4975]$. A tedious calculus exercise shows that the formula for $\q$ at \eqref{eq:q} is decreasing on the interval $[0.4950,0.4975]$, hence we may set $p=0.4975$ and solve for $q$ that makes $\q > 0.9999$ for this particular $p$-value. One can verify numerically, that $q \geq 0.9999998$ suffices, giving our claimed rectangle.

\section{Proof of \thref{prop:p}} \label{sec:prop}

\subsection*{Proof of (a)}
    Suppose that $q<1$. The number of frogs at each time step in $\FM(d,q,p)$ is dominated by a branching process that starts with one particle. At  discrete steps, each living particle will independently die with probability $1-q$, live with probability $q p$, or split into two particles with probability $q(1-p)$. The expected number of particles produced by a single particle after one step is thus $qp + 2q (1-p)$. This is less than or equal to $1$ if either $q \leq 1/2$ or if $p \geq 2- q^{-1}$. Hence, the frog model almost surely dies out for these parameter values.

\subsection*{Proof of (b)}
    Suppose that $p < 1/2$. The number of root visits in $\FM(d,q,p)$ is dominated by the number of root visits in the process that begins with all frogs awake. Letting $\rho = p /(1-p)$, the probability a $p$-biased random walk started at distance $n$ from the root ever reaches the root is $\rho^n$. Visiting the root for a frog at that distance requires living for at least $n$ steps. Hence the probability a frog at distance $n$ reaches the root is no more than $(q \rho)^n$. Using self-similarity for the $d^n$ frogs at distance $n$, the expected number of frogs to visit the root is bounded by $\sum_{n=0}^\infty (q \rho)^n d^n = \f{1}{1- q \rho d}$. This is finite so long as $q\rho d <1$, which is satisfied for the claimed $p < 1 / (1 + qd)$.

 \appendix 
 \section{Random walk excursion lengths}
The following lemma uses ideas from \cite[Section A]{hoffman2019infection}, where the law of simple random walk excursions for self-similar frog model paths were exactly characterized. Our case is slightly different due to bias and that steps in the recursive frog model allow some backtracking (the uniformly sampled ray going away from the root might overlap with a frog's previous steps taken towards the root). We also do not need an exact description since we only require a stochastic upper bound on the excursion lengths. 

\begin{lemma}\thlabel{lem:biased-walk-bounds}
Let $(X_t)_{t\ge0}$ be a nearest-neighbor random walk on $\mathbb Z$ with $X_0=n> 0$ that moves left with probability $0<p<1/2$ and right with probability $1-p$ at each step. Let $\rho = p/(1-p)$. We let $\Geo(r)$ denote the geometric distribution supported on the nonnnegative integers. 
\begin{enumerate}[label=\textup{(\alph*)}]
\item
For $\tau_{n-1}=\inf\{t\ge1: X_t=n-1\}$, it holds that
$$\f{\tau_n-1}{2} \preceq \Geo\left(r\right) \quad 
\text{ with } \quad r = \exp\left(-\frac{3(1 - 2p)^2}{2(1 + 4p)} \right).$$

\item Let $T_n = \max \{t \colon X_t =n\}$ be the time of the last visit to $n$. It holds that 
    $$\frac{T_n}{2}\ \preceq \Geo(s) \quad \text{ with } \quad s=(1-2p)r.$$
\item
Suppose the walk is conditioned to never go left of $n$ i.e.\ conditioned on the event $A=\{X_t \geq n \text{ for all $t \geq 0$}\}$. 
Let $T_n'$ denote the time of its last return to $n$. Then $T_n' \preceq T_n$.

\end{enumerate}
It follows that the length of each excursion between steps taken by the recursive frog model on $\mathbb T_d$ is stochastically dominated by the $2\Geo((1-p)(1-\rho)r) +1$ distribution. Moreover, these excursion lengths are independent between non-backtracking steps by the strong Markov property of random walk.

\end{lemma} 

\begin{proof}
\textup{(a)}\;
The event $\{\tau_{n-1}<\infty\}$ occurs with probability $\mathbf P_n(\tau_{n-1}<\infty)=p/(1-p)$.
Conditioning on $\{\tau_{n-1}<\infty\}$ corresponds to a Doob $h$--transform 
which reverses the drift of the walk. Hence, $\tau_n$ is distributed as the hitting time of a $n-1$ after reversing the drift. For $m \geq 0$, notice that $\tau_n \geq 2m +1$ requires at least $m$ rightward steps and thus
$$\P(\tau_n \geq 2m +1 \mid \{\tau_n < \infty\}) \leq \P( \Bin(2m+1, p) \geq m  )\leq \exp\left( - \f{ (\alpha-1)2p}{\f 23 + \f {2}{\alpha -1}}m \right) $$
with $\alpha = (2p)^{-1}$. 
The last inequality uses a standard concentration inequality for the Binomial distribution (see \cite[Proposition C.1]{hoffman2019infection}). Simplifying the final expression we obtain
$$\P^*(\tau_n \geq 2m +1 ) \leq \exp\left(-\frac{3(1 - 2p)^2}{2(1 + 4p)}
m \right).$$
Thus, conditional on $\{\tau_n < \infty\}$, we have $$\f{\tau_n-1}{2} \preceq \Geo\left(\exp\left(-\frac{3(1 - 2p)^2}{2(1 + 4p)} \right)\right).$$

\smallskip
\textup{(b)}\;
Suppose the random walk is at $n$. If its next jump is to $n-1$, then it will almost surely return to $n$ according to a biased random walk. If its next jump is to $n+1$, then it will return to $n$ with probability $\rho = p/(1-p)$. As in (a), conditional on this event it performs a random walk with the bias $1-p$ towards $n$. Thus, the total number of returns to $n$ is geometric with parameter $(1-p)(1-\rho)$. Each excursion corresponding to a failure is distributed as the excursions in (a). Thus, $T_n$ is stochastically dominated by a sum of $\Geo((1-p)(1- \rho))$ many independent $\Geo(r)$ random variables. Since $(1-p)(1-\rho) = 1-2p$, this gives the claimed result.

\smallskip
\textup{(c)}\; This follows from the fact that the random walk will almost surely have a last visit to $n$. Thus, $T_n$ counts at least as many excursions as $T_n'$. 
\end{proof}

\section*{Acknowledgements}
Ahmed was partially supported by NSF DMS Grants 2238272 and 2349366. Junge was partially supported by NSF DMS Grant 2238272.  Part of this research was conducted during the 2025 Baruch College Discrete Mathematics NSF Site REU. 

\bibliographystyle{alpha}
\bibliography{frog}

\end{document}